\documentclass{amsproc}

\usepackage{amssymb}

\newtheorem{theorem}{Theorem}[section]
\newtheorem{lemma}[theorem]{Lemma}
\newtheorem{Prop}[theorem]{Proposition}
\newtheorem{Qu}[theorem]{Question}

\theoremstyle{definition}
\newtheorem{example}[theorem]{Example}

\theoremstyle{remark}
\newtheorem{remark}[theorem]{Remark}

\numberwithin{equation}{section}

\newcommand\pf{\begin{proof}}
\newcommand\epf{\end{proof}}

\newcommand\Alg{\operatorname{Alg}}

\newcommand\CGal{\operatorname{CGal}}
\newcommand\Frac{\operatorname{Frac}}
\newcommand\Ker{\operatorname{Ker}}
\newcommand\Forms{\operatorname{Forms}}
\newcommand\id{{\operatorname{id}}}

\renewcommand\AA{\mathcal A}
\newcommand\BB{\mathcal B}
\newcommand\UU{\mathcal U}

\newcommand\ZZ{\mathbb Z}

\newcommand\mm{\mathfrak{m}}

\newcommand\gs{\mathfrak{s}}
\newcommand\gl{\mathfrak{l}}
\newcommand\eps{\varepsilon}

\newcommand\sw[1]{{}_{(#1)}}

\begin{document}

\title{Generic Hopf Galois extensions}

\author{Christian Kassel}
\address{Universit\'{e} de Strasbourg \& CNRS\\
Institut de Recherche Math\'{e}matique Avanc\'{e}e\\
7 rue Ren\'{e} Descartes\\
67084 Strasbourg, France}
\email{kassel@math.u-strasbg.fr}
\urladdr{www-irma.u-strasbg.fr/\raise-2pt\hbox{\~{}}kassel/}
\thanks{Partially funded by ANR Project BLAN07-3$_-$183390}

\keywords{Hopf algebra, Galois extension, twisted product, generic, cocycle, integrality}

\subjclass[2000]{Primary (16W30, 16S35, 16R50) 
Secondary (13B05, 13B22, 16E99, 58B32, 58B34, 81R50)}

\begin{abstract}
In previous joint work with Eli Aljadeff we attached a generic Hopf Galois extension~$\AA_H^{\alpha}$
to each twisted algebra~${}^{\alpha} H$ obtained from a Hopf algebra~$H$
by twisting its product with the help of a cocycle~$\alpha$.
The algebra~$\AA_H^{\alpha}$ is a flat deformation of~${}^{\alpha} H$
over a ``big" central subalgebra~$\BB_H^{\alpha}$
and can be viewed as the noncommutative analogue of a versal torsor in the sense of Serre.
After surveying the results on~$\AA_H^{\alpha}$ obtained with Aljadeff,
we establish three new results:
we present a systematic method to construct elements of  
the commutative algebra~$\BB_H^{\alpha}$,
we show that a certain important integrality condition 
is satisfied by all finite-dimensional
Hopf algebras generated by grouplike and skew-primitive elements,
and we compute~$\BB_H^{\alpha}$ in the case where $H$
is the Hopf algebra of a cyclic group.
\end{abstract}

\maketitle

\section*{Introduction}

In this paper we deal with associative algebras~${}^{\alpha} H$
obtained from a Hopf algebra~$H$
by twisting its product by a cocycle~$\alpha$. 
This class of algebras, which for simplicity we call \emph{twisted algebras},
coincides with the class of so-called 
\emph{cleft Hopf Galois extensions} of the ground field; 
classical Galois extensions and strongly group-graded algebras belong to this class. 
As has been stressed many times (see, e.g.,~\cite{S}), Hopf Galois
extensions can be viewed as noncommutative analogues of 
principal fiber bundles (also known as $G$-torsors), for which
the r\^ole of the structural group is played by a Hopf algebra. 
Hopf Galois extensions abound in the world of quantum groups
and of noncommutative geometry.
The problem of constructing systematically Hopf Galois extensions of a given algebra 
for a given Hopf algebra and of classifying them up to isomorphism
has been addressed in a number of papers over the last fifteen years;
let us mention \cite{Au, Bi, DT2, Gu, K, KS, Ma1, Ma11, Ma2, PvO, Sb1, Sb2}.
This list is far from being exhaustive, but gives a pretty good idea of the activity
on this subject.

A new approach to this problem was recently considered in~\cite{AK};
this approach mixes commutative algebra
with techniques from noncommutative algebra
such as \emph{polynomial identities}.
In particular, in that paper Eli Aljadeff and the author
attached two ``universal algebras" 
$\UU_H^{\alpha}$, $\AA_H^{\alpha}$ 
to each twisted algebra~${}^{\alpha} H$.
The algebra~$\UU_H^{\alpha}$, which was built out of
polynomial identities satisfied by~${}^{\alpha} H$,
was the starting point of~\emph{loc.\ cit}.
In the present paper we concentrate on the second algebra~$\AA_H^{\alpha}$ 
and survey the results obtained in~\cite{AK} from the point of view of this algebra.
In addition, we present here two new results, namely
Theorem~\ref{integralthm} and Proposition~\ref{identity0},
as well as a computation in Subsection~\ref{ZZZ}.

The algebra $\AA_H^{\alpha}$ is a ``generic" version of~${}^{\alpha} H$
and can be seen as a kind of universal Hopf Galois extension.
To construct~$\AA_H^{\alpha}$ 
we introduce the \emph{generic cocycle} cohomologous 
to the original cocycle~$\alpha$
and we consider
the commutative algebra~$\BB_H^{\alpha}$ generated by the values of
the generic cocycle and of its convolution inverse.
Then $\AA_H^{\alpha}$ is a cleft $H$-Galois extension of~$\BB_H^{\alpha}$.
We call $\AA_H^{\alpha}$ the \emph{generic Galois extension}
and $\BB_H^{\alpha}$ the \emph{generic base space}.
They satisfy the following remarkable properties.

Any ``form'' of~${}^{\alpha} H$ is obtained from~$\AA_H^{\alpha}$
by a specialization of~$\BB_H^{\alpha}$. 
Conversely, under an additional integrality condition,
any central specialization of~$\AA_H^{\alpha}$
is a form of~${}^{\alpha} H$.
Thus, the set of algebra morphisms~$\Alg(\BB_H^{\alpha},K)$
parametrizes the isomorphism classes of $K$-forms of~${}^{\alpha} H$
and $\AA_H^{\alpha}$ can be viewed as the noncommutative analogue 
of a \emph{versal deformation space} 
or a \emph{versal torsor} in the sense of Serre
(see~\cite[Chap.~I]{GMS}). 
We believe that such versal deformation spaces are of interest and deserve
to be computed for many Hopf Galois extensions.
Even when the Hopf algebra $H$ is a group algebra,
in which case our theory simplifies drastically, 
not many examples have been computed
(see \cite{AHN, AN} for results in this case).

Our approach also leads to the emergence of new interesting
questions on Hopf algebras such as Question~\ref{Question1} below.
We give a positive answer to this question for a class of
Hopf algebras that includes the finite-dimensional ones 
that are generated by grouplike and skew-primitive elements.

Finally we present a new systematic way
to construct elements of the generic base space~$\BB_H^{\alpha}$.
These elements are the images of certain universal noncommutative polynomials
under a certain tautological map. In the language
of polynomial identities, these noncommutative polynomials are central identities.

The paper is organized as follows. 
In Section~\ref{principal} we recall the concept of a Hopf Galois extension
and discuss the classification problem for such extensions.
In Section~\ref{twisted} we define Hopf algebra cocycles
and the twisted algebras~${}^{\alpha} H$. 
We construct the generic cocycle and
the generic base space~$\BB_H^{\alpha}$ in Section~\ref{generic cocycle};
we also compute~$\BB_H^{\alpha}$ when $H$ is the Hopf algebra of a cyclic group.
In Section~\ref{Sweedler}
we illustrate the theory with a nontrivial, still not too complicated example,
namely with the four-dimen\-sion\-al Sweedler algebra.
In Section~\ref{generic algebra} we define the 
generic Hopf Galois extension~$\AA_H^{\alpha}$ and 
state its most important properties. 
Some results of Section~\ref{generic algebra} hold under
a certain integrality condition; 
in Section~\ref{integrality} we prove that this condition is satisfied by
a certain class of Hopf algebras.
In Section~\ref{identities} we present the above-mentioned general method 
to construct elements of~$\BB_H^{\alpha}$.
The contents of Subsection~\ref{ZZZ} and
of Sections~\ref{integrality} and~\ref{identities} are new.

We consistently work over a fixed field~$k$, over which all our constructions 
will be defined. 
As usual, unadorned tensor symbols refer to the tensor product
of $k$-vector spaces. 
All algebras are assumed to be associative and unital,
and all algebra morphisms preserve the units.
We denote the unit of an algebra~$A$ by~$1_A$, 
or simply by~$1$ if the context is clear.
The set of algebra morphisms from an algebra~$A$ to an algebra~$B$
will be denoted by~$\Alg(A,B)$.

\pagebreak[2]

\section{Principal fiber bundles and Hopf Galois extensions}\label{principal}

\subsection{Hopf Galois extensions}\label{HopfGalois}
A principal fiber bundle involves a group~$G$ acting, say on the right, 
on a space~$X$ such that the map 
\[
X\times G \to X\times_Y X\, ;\; (x,g) \mapsto (x,xg)
\] 
is an isomorphism (in the category of spaces under consideration).
Here $Y$ represents some version of the quotient space~$X/G$
and $X\times_Y X$ the fiber product.

In a purely algebraic setting, the group $G$~is replaced by a Hopf algebra~$H$ 
with coproduct $\Delta : H\to H \otimes H$, co\"unit $\eps : H \to k$, 
and antipode $S : H \to H$.
In the sequel we shall make use of the Heyneman-Sweedler sigma notation
(see~\cite[Sect.~1.2]{Sw}): 
we write
\[
\Delta(x) = \sum_{(x)} x\sw1 \otimes x\sw2
\]
for the coproduct of $x \in H$ and
\[
\Delta^{(2)}(x) = \sum_{(x)} x\sw1 \otimes x\sw2 \otimes x_{(3)}
\]
for the iterated coproduct
$\Delta^{(2)} = (\Delta \otimes \id_H) \circ \Delta = (\id_H \otimes \Delta) \circ \Delta$,
and so on.

The $G$-space $X$ is replaced by an algebra~$A$ carrying the structure
of an $H$-comodule algebra. 
Recall that an algebra~$A$ is an $H$-\emph{comodule algebra}
if it has a right $H$-comodule structure whose coaction
$\delta : A \to A\otimes  H$ is an algebra morphism.

The space of \emph{co\"invariants} of an $H$-comodule algebra~$A$
is the subspace~$A^H$ of~$A$ defined by
\[
A^H = \{ a \in A \, | \, \delta(a)  = a \otimes 1\} \, .
\]
The subspace $A^H$ is a subalgebra and a subcomodule of~$A$.
We then say that $A^H\subset A$ is an $H$-\emph{extension}
or that $A$ is an $H$-extension of~$A^H$.
An $H$-extension is called \emph{central}
if $A^H$ lies in the center of~$A$.

An $H$-extension $B = A^H \subset A$ is said to be $H$-\emph{Galois}
if $A$ is faithfully flat as a left $B$-module and
the linear map $\beta : A\otimes_B A \to A\otimes H$
defined for $a$, $b\in A$ by
\[
a\otimes b \mapsto  (a\otimes 1_H) \, \delta(b)
\]
is bijective.
For a survey of Hopf Galois extensions, see \cite[Chap.~8]{M2}.

\begin{example}\label{graded-algebra}
The group algebra $H = k[G]$ of a group~$G$ is a Hopf algebra
with coproduct, co\"unit, and antipode respectively
given for all $g\in G$ by
\[
\Delta(g) = g\otimes g \, , \quad\eps(g) = 1\, , \quad S(g) = g^{-1}\, .
\]
This is a pointed Hopf algebra.
It is well known (see~\cite[Lemma~4.8]{BM}) that 
an $H$-comodule algebra~$A$ is the same as a $G$-graded algebra
\[
A = \bigoplus_{g\in G}\, A_g \, .
\]
The coaction $\delta : A \to A \otimes H$ is given by
$\delta(a) = a\otimes g$ for $a\in A_g$ and~$g\in G$.
We have $A^H = A_e$, where $e$ is the neutral element of~$G$.
Such a algebra is an $H$-Galois extension of~$A_e$ if and only if
the product induces isomorphisms 
\[
A_g \otimes_{A_e} A_h \cong A_{gh} \qquad(g,h\in G) \, .
\]
\end{example}

\subsection{(Uni)versal extensions}\label{classification}

An isomorphism $f: A \to A'$ of $H$-Galois extensions is
an isomorphism of the underlying $H$-comodule algebras,
i.e., an algebra morphism satisfying 
\[
\delta \circ f = (f\otimes \id_H) \circ \delta \, .
\]
Such an isomorphism necessarily sends $A^H$ onto~$A'{}^H$.

For any Hopf algebra~$H$ and any commutative algebra~$B$,
let $\CGal_H(B)$ denote the set of isomorphism classes of 
central $H$-Galois extensions of~$B$. 
It was shown in~\cite[Th.~1.4]{K} (see also~\cite[Prop.~1.2]{KS})
that any morphism $f: B\to B'$ of commutative algebras
induces a functorial map 
\[
f_* : \CGal_H(B) \to \CGal_H(B')
\]
given by $f_*(A) = B' \otimes_B A$
for all $H$-Galois extensions $A$ of~$B$.
The set $\CGal_H(B)$ is also a contravariant functor in~$H$
(see~\cite[Prop.~1.3]{KS}), but we will not make use of this fact here.

Recall that principal fiber bundles are classified as follows: there is a 
principal $G$-bundle $EG \to BG$, called the universal $G$-bundle
such that any principal $G$-bundle $X \to Y$ is obtained from pulling back 
the universal one along a continuous
map $f: Y \to BG$, which is unique up to homotopy.

By analogy, 
a \emph{(uni)versal $H$-Galois extension} would be a 
central $H$-Galois extension~$\BB_H \subset \AA_H$ such that for any
commutative algebra~$B$ and any central $H$-Galois extension $A$ of~$B$
there is a (unique) morphism of algebras $f: \BB_H \to B$ such that
$f_*(\AA_H) \cong A$. In other words, the map
\[
\Alg(\BB_H,B) \to \CGal_H(B)\, ; \, f \mapsto f_*(\AA_H)
\]
would be surjective (bijective). 
We have no idea if such (uni)versal $H$-Galois extensions
exist for general Hopf algebras. 

In the sequel, we shall only consider the case where $B=k$ and 
the $H$-Galois extensions of~$k$ are \emph{cleft}. Such
extensions coincide with the twisted algebras~${}^{\alpha} H$
introduced in the next section. To such an $H$-Galois extension
we shall associate a central $H$-Galois 
extension $\BB_H^{\alpha} \subset \AA_H^{\alpha}$,
such that the functor $\Alg(\BB_H^{\alpha},-)$ parametrizes 
the ``forms" of~${}^{\alpha} H$.
In this way we obtain an $H$-Galois extension that
is versal for a family of $H$-Galois extensions 
close to~${}^{\alpha} H$ in some appropriate \'etale-like Grothendieck topology.

\section{Twisted algebras}\label{twisted}

The definition of the twisted algebras~${}^{\alpha} H$
uses the concept of a cocycle, which we now recall.

\subsection{Cocycles}\label{cocycles}

Let $H$ be a Hopf algebra and $B$ a commutative algebra.
We use the following terminology.
A bilinear map $\alpha : H \times H \to B$ is
a \emph{cocycle of~$H$ with values in}~$B$ if
\begin{equation*}
\sum_{(x), (y)}\alpha(x\sw1,y\sw1)\, \alpha(x\sw2y\sw2,z)
= \sum_{(y), (z)} \alpha(y\sw1, z\sw1)\, \alpha(x,y\sw2z\sw2)
\end{equation*}
for all $x,y,z \in H$.
In the literature, what we call a cocycle is often
referred to as a ``left $2$-cocycle."

A bilinear map $\alpha : H \times H \to B$ is said
to be \emph{normalized} if
\begin{equation}\label{normalized}
\alpha(x,1_H)  = \alpha(1_H,x) = \varepsilon(x)\, 1_B
\end{equation}
for all $x\in H$. 

Two cocycles $\alpha, \beta : H \times H \to B$ are said
to be \emph{cohomologous} if 
there is an invertible linear map $\lambda : H \to B$ 
such that
\begin{equation}\label{cohomologous}
\beta(x,y) = \sum_{(x), (y)} \, 
\lambda(x_{(1)}) \, \lambda(y_{(1)}) \, \alpha(x_{(2)},y_{(2)}) \, 
\lambda^{-1}(x_{(3)}y_{(3)})
\end{equation}
for all $x, y \in H$.
Here ``invertible" means invertible with respect to the convolution product
and $\lambda^{-1} : H \to B$ denotes the inverse of~$\lambda$.
We write $\alpha \sim \beta$ if $\alpha, \beta$
are cohomologous cocycles.
The relation~$\sim$ is an equivalence relation on the set
of cocycles of~$H$ with values in~$B$.

\subsection{Twisted product}\label{Twisted products}

Let $H$ be a Hopf algebra, $B$ a commutative algebra,
and $\alpha : H \times H \to B$ be a normalized cocycle
with values in~$B$.
From now on, all cocycles are assumed to be
invertible with respect to the convolution product.

Let $u_H$ be a copy of the underlying vector space of~$H$.
Denote the identity map from $H$ to~$u_H$ by $x \mapsto u_x$ ($x\in H$).

We define the \emph{twisted algebra} $B\otimes {}^{\alpha} H$
as the vector space~$B\otimes u_H$ equipped 
with the associative product given by
\begin{equation}\label{twisted-multiplication}
(b\otimes u_x)  (c\otimes u_y)
= \sum_{(x), (y)}\, bc \, \alpha(x\sw1, y\sw1) \otimes  u_{x\sw2 y\sw2}
\end{equation}
for all $b$, $c\in B$ and $x$, $y \in H$.
Since $\alpha$ is a normalized cocycle,
$1_B \otimes u_1$ is the unit of~$B\otimes {}^{\alpha} H$.

The algebra $A = B\otimes {}^{\alpha} H$ is an $H$-comodule algebra
with coaction 
\[
\delta = \id_B \otimes \Delta : 
A = B\otimes H \to B\otimes H \otimes H = A \otimes H\, .
\] 
The subalgebra of co\"invariants of~$B\otimes {}^{\alpha} H$
coincides with~$B\otimes u_1$. 
Using~\eqref{normalized} and~\eqref{twisted-multiplication},
it is easy to check that this subalgebra lies in 
the center of~$B\otimes {}^{\alpha} H$.

It is well known that each twisted algebra~$B\otimes {}^{\alpha} H$ 
is a central $H$-Galois extension of~$B$.
Actually, the class of twisted algebra coincides with the 
class of so-called central cleft $H$-Galois extensions;
see \cite{BCM}, \cite{DT}, \cite[Prop.~7.2.3]{M2}.

An important special case of this construction occurs 
when $B = k$ is the ground field
and $\alpha : H \times H \to k$ is a cocycle of~$H$ with values in~$k$.
In this case, we simply call $\alpha$ a cocycle of~$H$.
Then the twisted algebra $k\otimes {}^{\alpha} H$,
which we henceforth denote by~${}^{\alpha} H$, coincides
with $u_H$ equipped with the associative product
\begin{equation*}
u_x \, u_y
= \sum_{(x), (y)}\,\alpha(x\sw1, y\sw1) \,  u_{x\sw2 y\sw2}
\end{equation*}
for all $x$, $y \in H$. 
The twisted algebras of the form~${}^{\alpha} H$
coincide with the so-called \emph{cleft $H$-Galois objects},
which are the cleft $H$-Galois extensions of the ground field~$k$.
We point out that for certain Hopf algebras~$H$ all $H$-Galois objects are cleft, e.g., if
$H$ is finite-dimensional or is a pointed Hopf algebra.

When $H = k[G]$ is the Hopf algebra of a group 
as in Example~\ref{graded-algebra}, then a $G$-graded algebra
$A = \bigoplus_{g\in G}\, A_g$ is an $H$-Galois object if and only if
$A_g A_h = A_{gh}$ for all $g,h\in G$ and $\dim A_g = 1$ for all $g\in G$.
Such an $H$-extension is cleft and thus isomorphic to~${}^{\alpha} H$
for some normalized invertible cocycle~$\alpha$.

\subsection{Isomorphisms of twisted algebras}\label{Isomorphisms}

By~\cite{BCM, Doi} there is
an isomor\-phism of $H$-comodule algebras between
the twisted algebras ${}^{\alpha}H$ and ${}^{\beta}H$ 
if and only if the cocycles $\alpha$ and $\beta$ are cohomologous
in the sense of~\eqref{cohomologous}.
It follows that the set of isomorphism classes of cleft $H$-Galois objects
is in bijection with the set of cohomology classes
of invertible cocycles of~$H$.

When the Hopf algebra~$H$ is cocommutative, then
the convolution product of two cocycles is a cocycle
and the set of cohomology classes of invertible cocycles of~$H$ is a group. 
This applies to the case $H = k[G]$; 
in this case the group of cohomology classes of invertible cocycles of~$k[G]$
is isomorphic to the cohomology group
$H^2(G,k^{\times})$ of the group~$G$ 
with values in the group~$k^{\times}$ of invertible elements of~$k$.

In general, the convolution product of two cocycles is \emph{not} a cocycle
and thus the set of cohomology classes of invertible cocycles is not a group.
One of the \emph{raisons d'\^etre} of the constructions presented here
and in~\cite{AK} lies in the lack of a suitable cohomology group
governing the situation.
We come up instead with the generic Galois extension defined below.

\section{The generic cocycle}\label{generic cocycle}

Let $H$ be a Hopf algebra and $\alpha : H\times H \to k$
an invertible normalized cocycle.

\subsection{The cocycle $\sigma$}

Our first aim is to construct a ``generic" cocycle of~$H$
that is cohomologous to~$\alpha$.

We start from the equation~\eqref{cohomologous}
\begin{equation*}
\beta(x,y) = \sum_{(x), (y)} \, 
\lambda(x_{(1)}) \, \lambda(y_{(1)}) \, \alpha(x_{(2)},y_{(2)}) \, 
\lambda^{-1}(x_{(3)}y_{(3)})
\end{equation*}
expressing that a cocycle $\beta$ is cohomologous to~$\alpha$, 
and the equation
\begin{equation}\label{lambda}
\sum_{(x)} \, \lambda(x_{(1)}) \, \lambda^{-1}(x_{(2)}) 
= \sum_{(x)} \, \lambda^{-1}(x_{(1)}) \, \lambda(x_{(2)})
= \eps(x) \, 1
\end{equation}
expressing that the linear form~$\lambda$ is invertible
with inverse~$\lambda^{-1}$.
To obtain the generic cocycle,
we proceed to mimic~\eqref{cohomologous}, replacing
the scalars $\lambda(x), \lambda^{-1}(x)$ respectively
by symbols~$t_x, t^{-1}_x$ satisfying~\eqref{lambda}.

Let us give a meaning to the symbols $t_x, t^{-1}_x$.
To this end we pick another copy $t_H$ of the underlying vector space of~$H$
and denote the identity map from $H$ to~$t_H$ by $x\mapsto t_x$ ($x\in H$).

Let $S(t_H)$ be the symmetric algebra over the vector space~$t_H$.
If $\{x_i\}_{i\in I}$ is a basis of~$H$, then $S(t_H)$ is
isomorphic to the polynomial algebra over the indeterminates $\{t_{x_i}\}_{i\in I}$.

By~\cite[Lemma~A.1]{AK}
there is a unique linear map $x \mapsto t^{-1}_x$
from~$H$ to the field of fractions~$\Frac S(t_H)$ of~$S(t_H)$
such that for all $x\in H$,
\begin{equation}\label{tt}
\sum_{(x)}\, t_{x_{(1)}} \, t^{-1}_{x_{(2)}}
= \sum_{(x)}\, t^{-1}_{x_{(1)}} \,  t_{x_{(2)}} = \eps(x) \, 1 \, .
\end{equation}
Equation~\eqref{tt} is the symbolic counterpart of~\eqref{lambda}.

Mimicking~\eqref{cohomologous}, we define a bilinear map
\[
\sigma : H \times H \to \Frac S(t_H)
\]
with values in the field of fractions $\Frac S(t_H)$ by the formula
\begin{equation}\label{sigma-def}
\sigma(x,y) = \sum_{(x), (y)}\, t_{x_{(1)}} \, t_{y_{(1)}} \,
\alpha(x_{(2)},y_{(2)}) \, t^{-1}_{x_{(3)} y_{(3)}}
\end{equation}
for all $x,y \in H$. 
The bilinear map $\sigma$ is a cocycle of~$H$ with values in~$\Frac S(t_H)$;
by definition, it is cohomologous to~$\alpha$.
We call $\sigma$ the \emph{generic cocycle attached to}~$\alpha$.

The cocycle $\alpha$ being invertible, so is~$\sigma$,
with inverse $\sigma^{-1}$ given for all $x,y \in H$ by
\begin{equation}\label{sigma^{-1}-def}
\sigma^{-1}(x,y) = \sum_{(x), (y)}\, t_{x_{(1)} y_{(1)}} \, 
\alpha^{-1}(x_{(2)},y_{(2)}) \, t^{-1}_{x_{(3)}}  \, t^{-1}_{y_{(3)}}\, ,
\end{equation}
where $\alpha^{-1}$ is the inverse of~$\alpha$.

In the case where $H = k[G]$ is the Hopf algebra of a group, the generic cocycle
and its inverse have the following simple expressions:
\begin{equation}\label{sigma-group}
\sigma(g,h) = \alpha(g,h)\, \frac{t_g \, t_h}{t_{gh}}
\quad\text{and}\quad
\sigma^{-1}(g,h) = \frac{1}{\alpha(g,h)}\, \frac{t_{gh}}{t_g \, t_h}
\end{equation}
for all $g$, $h \in G$.

\subsection{The generic base space}\label{algebraB}

Let $\BB_H^{\alpha}$ be the subalgebra of $\Frac S(t_H)$ 
generated by the values of the generic cocycle~$\sigma$ and 
of its inverse~$\sigma^{-1}$.
For reasons that will become clear in Section~\ref{generic algebra},
we call $\BB_H^{\alpha}$ the \emph{generic base space}.

Since $\BB_H^{\alpha}$ is a subalgebra of the field~$\Frac S(t_H)$, 
it is a domain and
the transcendence degree of the field of fractions of~$\BB_H^{\alpha}$
cannot exceed the dimension of~$H$.

In the case where $H$ is finite-dimensional,
$\BB_H^{\alpha}$ is a finitely generated algebra.
One can obtain a presentation of~$\BB_H^{\alpha}$ by generators and relations 
using standard monomial order techniques of commutative algebra.

\subsection{A computation}\label{ZZZ}

Let $H = k[\ZZ]$ be the Hopf algebra of the group~$\ZZ$ of integers.
We write~$\ZZ$ multiplicatively and identify its elements with the powers
$x^m$ of a variable~$x$ ($m \in \ZZ$).

We take $\alpha$ to be the trivial cocycle, i.e., $\alpha(g,h) = 1$ for all $g,h \in\ZZ$
(this is no restriction since $H^2(\ZZ,k^{\times}) = 0$).
In this case the symmetric algebra $S(t_H)$ coincides with the
polynomial algebra $k[t_m \, |\, m\in \ZZ]$.
Set $y_m = t_m/t_1^m$ for each $m\in \ZZ$. 
We have $y_1 = 1$ and $y_0 = t_0$.

By \eqref{sigma-group}, the generic cocycle is given by
\[
\sigma(x^m, x^n) = \frac{t_m t_n}{t_{m+n}} 
\]
for all $m,n \in \ZZ$. 
This can be reformulated as
\[
\sigma(x^m, x^n) = \frac{y_m y_n}{y_{m+n}} \, .
\]
The inverse of~$\sigma$ is given by
\[
\sigma^{-1}(x^m, x^n) = \frac{1}{\sigma(x^m, x^n)} 
=  \frac{y_{m+n}}{y_m y_n} \, .
\]

A simple computation yields the following expressions of~$y_m$
in the values of~$\sigma$ and~$\sigma^{-1}$:
\begin{equation*}
y_m = 
\begin{cases}
\sigma^{-1}(x^{m-1}, x) \, \sigma^{-1}(x^{m-2}, x) \cdots \sigma^{-1}(x, x) 
& \text{if $m \geq 2$} \, ,\\
1 & \text{if $m=1$} \, ,\\
\sigma(x^0,x^0) & \text{if $m=0$} \, ,\\
\sigma(x^m,x^{-m}) \sigma(x^{-m-1}, x) \sigma(x^{-m-2}, x) \cdots \sigma(x, x)\sigma(x^0,x^0)
& \text{if $m\leq -1$} \, .
\end{cases}
\end{equation*}
It follows that the elements $y_m^{\pm 1}$ belong to~$\BB_H^{\alpha}$
for all~$m\in \ZZ - \{1\}$ and generate this algebra.
It is easy to check that the family $(y_m)_{m\neq 1}$ is
algebraically independent, so that
$\BB_H^{\alpha}$ is the Laurent polynomial algebra
\[
\BB_H^{\alpha} = k[\, y_m^{\pm 1} \, |\, m\in \ZZ - \{1\} \, ] \, .
\]
We deduce the algebra isomorphism
\begin{equation}\label{B-Z}
k[t_m^{\pm 1} \, |\, m\in \ZZ] \cong \BB_H^{\alpha}[t_1^{\pm 1}]  \, .
\end{equation}

If in the previous computations
we replace $\ZZ$ by the cyclic group $\ZZ/N$, where $N$ is some
integer $N \geq 2$, then
the algebra $\BB_H^{\alpha}$ is again a Laurent polynomial algebra:
\[
\BB_H^{\alpha} = k[y_0^{\pm 1}, y_2^{\pm 1}, \ldots, y_N^{\pm 1}]\, 
\]
where $y_0, y_2, \ldots, y_{N-1}$ are defined as above
and $y_N = t_0/t_1^N$.
In this case, the algebra
$k[t_0^{\pm 1}, t_1^{\pm 1}, \ldots, t_{N-1}^{\pm 1}]$
is an integral extension of~$\BB_H^{\alpha}$:
\[
k[t_0^{\pm 1}, t_1^{\pm 1}, \ldots, t_{N-1}^{\pm 1}]
\cong \BB_H^{\alpha}[t_1]/(t_1^N - y_0/y_N)  \, .
\]

\section{The Sweedler algebra}\label{Sweedler}

We now illustrate the constructions of Section~\ref{generic cocycle}
on Sweedler's four-dimen\-sion\-al Hopf algebra.
We assume in this section that the characteristic of the ground
field~$k$ is different from~2.

The \emph{Sweedler algebra} $H_4$ is the
algebra generated by two elements $x$, $y$ subject to the relations
\begin{equation*}
x^2 = 1  \, , \quad xy + yx = 0 \, , \quad y^2 = 0 \, .
\end{equation*}
It is four-dimensional. As a basis of~$H_4$ 
we take the set $\{1, x, y, z\}$, where $z = xy$.

The algebra~$H_4$ carries the structure of a Hopf algebra with
coproduct, co\"unit, and antipode given by
\[
\begin{array}{rclcrcl}
\Delta(1) & = & 1 \otimes 1 \, ,  &\; &  \Delta(x) & = &  x \otimes x \, ,\\  
\Delta(y) & = & 1 \otimes y + y \otimes x \, ,  & \; &   \Delta(z) & = &  x \otimes z + z \otimes 1 \, ,\\  
\eps(1) & = & \eps(x) =   1 \, ,  & \; &   \eps(y) & = &   \eps(z) = 0 \, ,\\  
S(1) & = & 1 \, ,  &\; &  S(x) & = &  x  \, ,\\
S(y) & = & z \, ,  &\; & S(z) & = &  -y  \, .
\end{array}
\]

By definition, the symbols $t_x$ and $t^{-1}_x$ satisfy the equations
\[
\begin{array}{rclcrcl}
t_1 \, t^{-1}_1 & = & 1 \, ,  &\; &  t_x \, t^{-1}_x & = &  1 \, ,\\  
t_1 \, t^{-1}_y + t_y \, t^{-1}_x & = & 0 \, ,  & \; &   
t_x \, t^{-1}_z + t_z \, t^{-1}_1 & = &  0 \, .
\end{array}
\]
Hence,
\begin{equation*}
t^{-1}_1 = \frac{1}{t_1} \, ,\quad  
t^{-1}_x = \frac{1}{t_x} \, ,\quad  
t^{-1}_y =  - \frac{t_y}{t_1t_x}\ , \quad  
t^{-1}_z = - \frac{t_z} {t_1 t_x}  \, .
\end{equation*}

Masuoka~\cite{Ma1} showed that any cleft 
$H_4$-Galois object has, up to isomorphism, the following presentation:
\[
{}^{\alpha}H_4
= k \langle \, u_x, u_y\, |\, u_x^2 = a\, , \;\; 
u_x u_y + u_y u_x = b \, , \;\; u_y^2 = c \, \rangle
\]
for some scalars $a$, $b$, $c$ with $a\neq 0$.
To indicate the dependence on the parameters $a$, $b$~$c$, we
denote ${}^{\alpha} H_4$ by~$A_{a,b,c}$. 

It is easy to check that the center of~$A_{a,b,c}$ is trivial for all
values of $a$, $b$,~$c$.
Moreover, the algebra $A_{a,b,c}$ is simple if and only $b^2-4ac \neq 0$.
If $b^2-4ac =0$, then $A_{a,b,c}$ is isomorphic as an algebra to~$H_4$;
the latter is not semisimple 
since the two-sided ideal generated by~$y$ is nilpotent.

The generic cocycle $\sigma$ attached to~$\alpha$ has the following values:
\begin{eqnarray*}
\sigma(1,1) & = & \sigma(1,x) = \sigma(x,1) = t_1\, ,\\
\sigma(1,y) & = & \sigma(y,1) = \sigma(1,z) = \sigma(z,1) = 0 \, ,\\
\sigma(x,x) & = & at_x^2 t_1^{-1}, \\
\sigma(y,y) & = & \sigma(z,y) = - \sigma(y,z)
= (a t_y^2 + b t_1 t_y + c t_1^2) \, t_1^{-1},\\
\sigma(x,y) & = &  - \sigma(x,z) = (a t_x t_y - t_1t_z ) \, t_1^{-1}, \\
\sigma(y,x) & = & \sigma(z,x) =  (bt_1t_x + a t_x t_y + t_1t_z) \, t_1^{-1}, \\
\sigma(z,z) & = & -(t_z^2 + b t_x t_z + ac t_x^2) \, t_1^{-1},
\end{eqnarray*}
The values of the inverse $\sigma^{-1}$ are equal to the values of~$\sigma$
possibly divided by positive powers of~$t_1$ 
and of~$\sigma(x,x) = at_x^2 t_1^{-1}$.

By definition, $\BB_{H_4}^{\alpha}$ is
the subalgebra of~$\Frac S(t_{H_4})$ generated
by the values of $\sigma$ and~$\sigma^{-1}$.
If we set
\begin{equation}\label{ERSTU}
\begin{aligned}
& E = t_1 \, , \quad  R = a \, t_x^2 \, , \quad  S = a \, t_y^2 + b \, t_1t_y + c \, t_1^2 \, , \\
& T = t_x \, (2a \, t_y + b \, t_1) \, ,  \quad  U = a \, t_x^2 \, (2\, t_z + b \, t_x) \, , 
\end{aligned}
\end{equation}
then we can reformulate the above (nonzero) values of~$\sigma$ as follows:
\begin{align*}
\sigma(1,1) & =  \sigma(1,x) = \sigma(x,1) = E\, ,\\
\sigma(x,x) & =  \frac{R}{E} \, , \displaybreak[1] \\
\sigma(y,y) & =  \sigma(z,y) = - \sigma(y,z) = \frac{S}{E} \, , \displaybreak[1] \\
\sigma(x,y) & =   - \sigma(x,z) = \frac{RT - EU}{2ER} \, , \displaybreak[1] \\
\sigma(y,x) & =  \sigma(z,x) =  \frac{RT + EU}{2ER} \, , \\
\sigma(z,z) & =  \frac{a \, U^2 - (b^2 - 4ac) \, R^3}{4a \, ER^2} \, .
\end{align*}
From the previous equalities we conclude
that $E^{\pm 1}$, $R^{\pm 1}$, $S$, $T$,~$U$ belong to~$\BB_{H_4}^{\alpha}$ 
and that they generate it as an algebra.

In \cite[Sect.~10]{AK} we obtained the following presentation of~$\BB_{H_4}^{\alpha}$ 
by generators and relations.

\begin{theorem}\label{H4universal}
We have
\[
\BB_{H_4}^{\alpha} \cong k[E^{\pm 1}, R^{\pm 1}, S, T, U]/(P_{a,b,c})\, ,
\]
where
\[
P_{a,b,c} = T^2 - 4 RS - \frac{b^2- 4ac}{a} \, E^2 R \, .
\]
\end{theorem}

It follows from the previous theorem
that the algebra morphisms from~$\BB_{H_4}^{\alpha}$
to a field~$K$ 
containing~$k$
are in one-to-one correspondence with the quintuples
$({e},{r}, {s}, {t},{u}) \in K^5$ verifying ${e} \neq 0$, ${r} \neq 0$, and
the equation
\begin{equation}\label{H4relation}
{t}^2 - 4 {r} {s} 
=  \frac{b^2- 4ac}{a} \, {e}^2 {r} \, .
\end{equation}
In other words, the set of $K$-points of~$\BB_{H_4}^{\alpha}$ 
is the hypersurface of equation~\eqref{H4relation}
in~$K^{\times} \times K^{\times} \times K \times K \times K$.

\section{The generic Galois extension}\label{generic algebra}

As in Section~\ref{generic cocycle}, we consider a Hopf algebra~$H$ 
and an invertible normalized cocycle $\alpha : H\times H \to k$.
Let ${}^{\alpha} H$ be the corresponding twisted algebra.

\subsection{The algebra~$\AA_H^{\alpha}$}\label{algebraA}

By the definition of the commutative algebra~$\BB_H^{\alpha}$
given in Section~\ref{algebraB},
the generic cocycle~$\sigma$ takes its values in~$\BB_H^{\alpha}$. 
Therefore we may apply the construction of Section~\ref{Twisted products}
and consider the twisted algebra
\begin{equation*}
\AA_H^{\alpha} = \BB_H^{\alpha} \otimes {}^{\sigma} H \, .
\end{equation*}
The product of~$\AA_H^{\alpha}$ is given
for all $b$, $c\in \BB_H^{\alpha}$ and $x$, $y \in H$ by
\begin{equation*}
(b\otimes u_x)  (c\otimes u_y)
= \sum_{(x), (y)}\, bc \, \sigma(x\sw1, y\sw1) \otimes  u_{x\sw2 y\sw2} \, .
\end{equation*}
We call~$\AA_H^{\alpha}$ the \emph{generic Galois extension}
attached to the cocycle~$\alpha$.

The subalgebra of co\"invariants of~$\AA_H^{\alpha}$
is equal to~$\BB_H^{\alpha} \otimes u_1$; 
this subalgebra is central in~$\AA_H^{\alpha}$.
Therefore, $\AA_H^{\alpha}$ is a 
central cleft $H$-Galois extension of~$\BB_H^{\alpha}$.

By~\cite[Prop.~5.3]{AK}, there is an algebra morphism
$\chi_0 : \BB_H^{\alpha} \to k$ such that
\[
\chi_0 \bigl( \sigma(x,y) \bigr) = \alpha(x,y)
\quad\text{and}\quad
\chi_0 \bigl( \sigma^{-1}(x,y) \bigr) = \alpha^{-1}(x,y)
\]
for all $x,y \in H$. 
Consider the maximal ideal $\mm_0 = \Ker (\chi_0 : \BB_H^{\alpha} \to k)$ 
of~$\BB_H^{\alpha}$. 
According to~\cite[Prop.~6.2]{AK},
there is an isomorphism of $H$-comodule algebras
\begin{equation*}
\AA_H^{\alpha}/\mm_0 \, \AA_H^{\alpha} \cong {}^{\alpha} H \, .
\end{equation*}
Thus, $\AA_H^{\alpha}$ is a \emph{flat deformation}
of~${}^{\alpha} H$ over the commutative algebra~$\BB_H^{\alpha}$.

Certain properties of~${}^{\alpha} H$ lift to 
the generic Galois extension~$\AA_H^{\alpha}$
such as the one recorded in the following result of~\cite{AK},
where $\Frac \BB_H^{\alpha}$ stands for
the field of fractions of~$\BB_H^{\alpha}$.

\begin{theorem}\label{AAsemisimple}
Assume that the ground field~$k$ is of characteristic zero and
the Hopf algebra~$H$ is finite-dimen\-sional.
If the algebra ${}^{\alpha} H$ is simple (resp.\ semisimple),
then so is 
\[
\Frac \BB_H^{\alpha} \otimes_{\BB_H^{\alpha}} \AA_H^{\alpha}
= \Frac \BB_H^{\alpha} \otimes {}^{\sigma} H \, .
\]
\end{theorem}

\subsection{Forms}\label{forms}

We have just observed that ${}^{\alpha} H \cong \AA_H^{\alpha}/\mm_0 \, \AA_H^{\alpha}$ 
for some maximal ideal~$\mm_0$ of~$\BB_H^{\alpha}$.
We may now wonder what can be said of the other \emph{central specializations}
of~$\AA_H^{\alpha}$, that is of the quotients $\AA_H^{\alpha}/\mm \, \AA_H^{\alpha}$,
where $\mm$ is an arbitrary maximal ideal of~$\BB_H^{\alpha}$.
To answer this question, we need the following terminology.

Let $\beta : H \times H \to K$ be a normalized invertible cocycle
with values in a field~$K$ containing the ground field~$k$.
We say that the twisted $H$-comodule algebra $K \otimes {}^{\beta} H$
is a \emph{$K$-form} of ${}^{\alpha} H$ if
there is a field~$L$ containing~$K$ and 
an $L$-linear isomorphism of~$H$-comodule algebras
\[
L \otimes_K (K \otimes {}^{\beta} H) \cong L\otimes_k {}^{\alpha} H \, .
\]

We now state two theorems relating forms of~${}^{\alpha} H$
to central specializations of
the generic Galois extension~$ \AA_H^{\alpha}$.
For proofs, see \cite[Sect.~7]{AK}.

\begin{theorem}\label{formthm2}
For any $K$-form $K \otimes {}^{\beta} H$ of ${}^{\alpha} H$,
where $\beta : H \times H \to K$ is a normalized invertible cocycle
with values in an extension~$K$ of~$k$,
there exist an algebra morphism $\chi : \BB_H^{\alpha} \to K$ 
and a $K$-linear isomorphism of~$H$-comodule algebras
\[
K_{\chi} \otimes_{\BB_H^{\alpha}} \AA_H^{\alpha} 
\cong K \otimes {}^{\beta} H \, .
\]
\end{theorem}

Here $K_{\chi}$ stands for~$K$ equipped with the  
$\BB_H^{\alpha}$-module structure induced by
the algebra morphism $\chi : \BB_H^{\alpha} \to K$.
We have 
\[
K_{\chi} \otimes_{\BB_H^{\alpha}} \AA_H^{\alpha} 
\cong \AA_H^{\alpha}/\mm_{\chi} \, \AA_H^{\alpha} \, ,
\]
where $\mm_{\chi} = \Ker(\chi : \BB_H^{\alpha} \to K)$.

There is a converse to Theorem~\ref{formthm2}; 
it requires an additional condition.

\begin{theorem}\label{formthm}
If $\Frac S(t_H)$ is integral over the subalgebra~$\BB_H^{\alpha}$, 
then for any field~$K$ containing~$k$ and any algebra morphism
$\chi : \BB_H^{\alpha} \to K$, the $H$-comodule $K$-algebra
$K_{\chi} \otimes_{\BB_H^{\alpha}} \AA_H^{\alpha} 
= \AA_H^{\alpha}/ \mm_{\chi} \, \AA_H^{\alpha}$
is a $K$-form of~${}^{\alpha} H$.
\end{theorem}

It follows that if $\Frac S(t_H)$ is integral over~$\BB_H^{\alpha}$, 
then the map 
\begin{eqnarray*}
\Alg(\BB_H^{\alpha}, K) & \longrightarrow & K\text{-}\Forms({}^{\alpha} H) \\
\chi &\longmapsto& K_{\chi} \otimes _{\BB_H^{\alpha}} \AA_H^{\alpha}
= \AA_H^{\alpha}/ \mm_{\chi} \, \AA_H^{\alpha}
\end{eqnarray*}
is a \emph{surjection} from the set of algebra morphisms $\BB_H^{\alpha} \to K$ 
to the set of isomorphism classes of $K$-forms of~${}^{\alpha} H$.
Thus the set $\Alg(\BB_H^{\alpha}, K)$ parametrizes 
the $K$-forms of~${}^{\alpha} H$.
Using terminology of singularity theory, we say 
that the Galois extension
$\BB_H^{\alpha} \subset \AA_H^{\alpha}$ is a
\emph{versal deformation space} for the forms of~${}^{\alpha} H$
(we would call this space universal if the above surjection was bijective).

By Theorem~\ref{AAsemisimple}, the central localization
$\Frac \BB_H^{\alpha} \otimes_{\BB_H^{\alpha}} \AA_H^{\alpha}$
is a simple algebra if the algebra ${}^{\alpha} H$ is simple.
Under the integrality condition above, we have the following related result
(see~\cite[Th.~7.4]{AK}).

\begin{theorem}\label{Azumaya}
If $\Frac S(t_H)$ is integral over~$\BB_H^{\alpha}$
and if the algebra ${}^{\alpha} H$ is simple,
then $\AA_H^{\alpha}$ is an Azumaya algebra with center~$\BB_H^{\alpha}$.
\end{theorem}

This means that $\AA_H^{\alpha}/\mm \, \AA_H^{\alpha}$ is a simple algebra
for any maximal ideal~$\mm$ of~$\BB_H^{\alpha}$.
For instance, any full matrix algebra with entries in a commutative
algebra is Azumaya.

\begin{example}
For the Sweedler algebra~$H_4$, we proved in~\cite[Sect.~10]{AK}
that $\AA_{H_4}^{\alpha}$ is given as an algebra by
\begin{equation}
\AA_{H_4}^{\alpha} \cong \BB_{H_4}^{\alpha}
\langle\, X,Y \, \rangle/
(X^2-R \, , \, Y^2-S \, , \, XY+YX - T)\, ,
\end{equation}
where $\BB_{H_4}^{\alpha}$ is as in Theorem~\ref {H4universal}
and the elements $R,S,T$ of~$\BB_{H_4}^{\alpha}$ are defined by~\eqref{ERSTU}.
As an $\BB_{H_4}^{\alpha}$-module, $\AA_{H_4}^{\alpha}$ is free with basis $\{1,X,Y,XY\}$.
\end{example}

\section{The integrality condition}\label{integrality}

In view of Theorem~\ref{formthm} it is natural to ask the following question.

\begin{Qu}\label{Question1}
Under which condition on the pair~$(H,\alpha)$ is 
$\Frac S(t_H)$ integral
over the subalgebra~$\BB_H^{\alpha}$?
\end{Qu}

Question~\ref{Question1} has a negative answer in the case where
$H= k[\ZZ]$ and $\alpha$ is the trivial cocycle.
Indeed, it follows from~\eqref{B-Z}
that $\Frac S(t_H)$ is then a pure transcendental extension
(of degree~one) of the field of fractions of~$\BB_H^{\alpha}$.

We give a positive answer in the following important case.

\begin{theorem}\label{integralthm}
Let $H$ be a Hopf algebra generated as an algebra by a set~$\Sigma$ of
grouplike and skew-primitive elements such that
the grouplike elements of~$\Sigma$ are of finite order and
generate the group of grouplike elements of~$H$
and such that each skew-primitive element of~$\Sigma$
generates a finite-dimensional subalgebra of~$H$.
Then $\Frac S(t_H)$ is integral over the subalgebra~$\BB_H^{\alpha}$
for every cocycle~$\alpha$ of~$H$.
\end{theorem}

Theorem~\ref{integralthm} implies a positive answer to Question~\ref{Question1}
for any finite-dimensional Hopf algebra generated by grouplike and
skew-primitive elements.
It is conjectured that all finite-dimensional \emph{pointed} Hopf algebras 
are generated by grouplike and skew-primitive elements;
if this conjecture holds, then Question~\ref{Question1} has a positive answer 
for any finite-dimensional Hopf algebra that is pointed.

Recall that $g\in H$ is \emph{grouplike} if 
$\Delta(g) = g \otimes g$;
it then follows that $\eps(g) = 1$.
The inverse of a grouplike element and the product of two 
grouplike elements are grouplike.
An element $x\in H$ is \emph{skew-primitive} if 
\begin{equation}\label{skew}
\Delta(x) = g \otimes x + x \otimes h
\end{equation} 
for some grouplike elements $g,h \in H$;
this implies $\eps(x) = 0$. 
The product of a skew-primitive element by
a grouplike element is skew-primitive.

In order to prove Theorem~\ref{integralthm}, we need the following lemma.

\begin{lemma}\label{txtytz}
If $x^{[1]}, \ldots, x^{[n]}$ are elements of~$H$, then
\begin{align*}
\begin{split}
t_{x^{[1]} \cdots x^{[n]}} 
& =  \sum_{x^{[1]}, \ldots, x^{[n]}}\, 
\sigma^{-1}(x^{[1]}_{(1)} \cdots x^{[n-1]}_{(1)}, x^{[n]}_{(1)}) \, \times \\
& \hskip 40pt \times
\sigma^{-1}(x^{[1]}_{(2)} \cdots x^{[n-2]}_{(2)}, x^{[n-1]}_{(2)}) \cdots
\sigma^{-1}(x^{[1]}_{(n-1)} , x^{[2]}_{(n-1)})  \, \times \\
& \hskip 40pt \times t_{x^{[1]}_{(n)}} \, t_{x^{[2]}_{(n)}} \cdots  
t_{x^{[n-1]}_{(3)}} \, t_{x^{[n]}_{(2)}} \, \times \\
& \hskip 40pt \times \alpha(x^{[1]}_{(n+1)}, x^{[2]}_{(n+1)}) \cdots
\alpha(x^{[1]}_{(2n-2)} \, x^{[2]}_{(2n-2)} \cdots x^{[n-2]}_{(6)}, x^{[n-1]}_{(4)}) \, \times \\
& \hskip 40pt \times 
\alpha(x^{[1]}_{(2n-1)} \, x^{[2]}_{(2n-1)} \cdots x^{[n-1]}_{(5)}, x^{[n]}_{(3)}) \, .
\end{split}
\end{align*}
\end{lemma}

\pf
We prove the formula by induction on~$n$.
When $n=2$ it reduces to
\begin{equation}\label{txty}
t_{xy} =
\sum_{(x), (y)}\, \sigma^{-1}(x_{(1)},y_{(1)}) \, 
t_{x_{(2)}}\, t_{y_{(2)}} \,  \alpha(x_{(3)},y_{(3)}) 
\end{equation}
for $x,y \in H$. Let us first prove~\eqref{txty}.
By~\eqref{sigma^{-1}-def} the right-hand side of~\eqref{txty} is equal to
\begin{eqnarray*}
&& \sum_{(x), (y)}\, t_{x_{(1)} y_{(1)}} \, 
\alpha^{-1}(x_{(2)},y_{(2)}) \, t^{-1}_{x_{(3)}} 
\underbrace{t^{-1}_{y_{(3)}} \, t_{y_{(4)}} } \, t_{x_{(4)}}\,  \alpha(x_{(5)},y_{(5)}) \\
& = & \sum_{(x), (y)}\, t_{x_{(1)} y_{(1)}} \, 
\alpha^{-1}(x_{(2)},y_{(2)}) \, 
\underbrace{t^{-1}_{x_{(3)}} \, t_{x_{(4)}} }\,  \alpha(x_{(5)},y_{(3)}) \\
& = & \sum_{(x), (y)}\, t_{x_{(1)} y_{(1)}} \, 
\underbrace{\alpha^{-1}(x_{(2)},y_{(2)}) \, \,  \alpha(x_{(3)},y_{(3)}) } \\
& = & \sum_{(x), (y)}\, t_{x_{(1)} y_{(1)}} \, 
\eps(x_{(2)}) \, \eps(y_{(2)}) \\
& = & \sum_{(x), (y)}\, t_{x_{(1)} y_{(1)}} \, 
\eps(x_{(2)} y_{(2)}) 
= t_{xy} \, .
\end{eqnarray*}

Let us assume that Lemma~\ref{txtytz} holds for all $n$-tuples of~$H$
and consider a sequence $(x^{[1]}, x^{[2]}, \ldots, x^{[n+1]})$ of $n+1$ elements  of~$H$.
By the induction hypothesis and by~\eqref{txty}, 
$t_{x^{[1]}x^{[2]} \cdots x^{[n+1]}} = t_{(x^{[1]}x^{[2]}) \cdots x^{[n+1]}}$ is equal to
\begin{align*}
&  \sum_{x^{[1]}, x^{[2]}, \ldots, x^{[n+1]}}\, 
\sigma^{-1} \bigl( (x^{[1]}_{(1)} x^{[2]}_{(1)}) \cdots x^{[n]}_{(1)}, x^{[n+1]}_{(1)} \bigr) \, \times \\
& \hskip 40pt \times
\sigma^{-1}( \bigl( x^{[1]}_{(2)} x^{[2]}_{(2)}) \cdots x^{[n-1]}_{(2)}, x^{[n]}_{(2)} \bigr) 
\cdots \sigma^{-1}(x^{[1]}_{(n-1)} x^{[2]}_{(n-1)}, x^{[3]}_{(n-1)}) \, \times \\
& \hskip 40pt \times t_{x^{[1]}_{(n)} x^{[2]}_{(n)}} \, t_{x^{[3]}_{(n)}} \cdots  
t_{x^{[n]}_{(3)}} \, t_{x^{[n+1]}_{(2)}} \, \times  \\
& \hskip 40pt \times \alpha(x^{[1]}_{(n+1)} x^{[2]}_{(n+1)}, x^{[3]}_{(n+1)}) \cdots \\
& \hskip 40pt \cdots
\alpha \bigl( (x^{[1]}_{(2n-2)} \, x^{[2]}_{(2n-2)})\, x^{[3]}_{(2n-2)} \cdots 
x^{[n-1]}_{(6)}, x^{[n]}_{(4)} \bigr) \, \times \\
& \hskip 40pt \times 
\alpha \bigl( (x^{[1]}_{(2n-1)} \, x^{[2]}_{(2n-1)}) \, x^{[3]}_{(2n-1)} 
\cdots x^{[n]}_{(5)}, x^{[n+1]}_{(3)} \bigr)  \displaybreak[1]\\
& =  \sum_{x^{[1]}, x^{[2]}, \ldots, x^{[n+1]}}\, 
\sigma^{-1}(x^{[1]}_{(1)} x^{[2]}_{(1)} \cdots x^{[n]}_{(1)}, x^{[n+1]}_{(1)}) \,
\sigma^{-1}(x^{[1]}_{(2)} x^{[2]}_{(2)} \cdots x^{[n-1]}_{(2)}, x^{[n]}_{(2)}) \cdots \\
& \hskip 40pt \cdots \sigma^{-1}(x^{[1]}_{(n-1)} x^{[2]}_{(n-1)}, x^{[3]}_{(n-1)})
\, \sigma^{-1}(x^{[1]}_{(n)} , x^{[2]}_{(n)}) \, \times \\
& \hskip 40pt \times t_{x^{[1]}_{(n+1)}} \, t_{x^{[2]}_{(n+1)}} \, t_{x^{[3]}_{(n)}} \cdots  
t_{x^{[n]}_{(3)}} \, t_{x^{[n+1]}_{(2)}} \, \times  \\
& \hskip 40pt \times 
\alpha(x^{[1]}_{(n+2)} , x^{[2]}_{(n+2)})  \, 
\alpha(x^{[1]}_{(n+3)} x^{[2]}_{(n+3)}, x^{[3]}_{(n+1)}) \cdots \\
& \hskip 40pt 
\cdots \alpha(x^{[1]}_{(2n)} \, x^{[2]}_{(2n)}\, x^{[3]}_{(2n-2)} \cdots 
x^{[n-1]}_{(6)}, x^{[n]}_{(4)}) \, \times \\
& \hskip 40pt \times 
\alpha(x^{[1]}_{(2n+1)} \, x^{[2]}_{(2n+1)} \, x^{[3]}_{(2n-1)} \cdots x^{[n]}_{(5)}, x^{[n+1]}_{(3)}) \, ,
\end{align*}
which is the desired formula for $n+1$ elements.
\epf

\begin{proof}[Proof of Theorem~\ref{integralthm}]
Let $A$ be the integral closure of~$ \BB_H^{\alpha}$ in~$\Frac S(t_H)$.
To prove the theorem it suffices to
establish that each generator~$t_z$ of $S(t_H)$ belongs to~$A$.

We start with the unit of~$H$.
By~\cite[Lemma~5.1]{AK}, $t_1 = \sigma(1,1)$. Thus $t_1$
belongs to~$\BB_H^{\alpha}$, hence to~$A$.

Let $g$ be a grouplike element of the generating set~$\Sigma$. 
By hypothesis, there is an integer $n\geq 2$ such that $g^n = 1$.
We apply Lemma~\ref{txtytz} to $x^{[1]} = \cdots = x^{[n]} = g$.
Since any iterated coproduct $\Delta^{(p)}$ applied to~$g$ yields
\begin{equation}\label{gitere}
\Delta^{(p)}(g) = g\otimes g \otimes \cdots \otimes g\, ,
\end{equation}
where the right-hand side is the tensor product of $p$~copies of~$g$,
we obtain
\begin{multline}\label{ggg}
t_{g^n} = 
\sigma^{-1}(g^{n-1}, g) \, \sigma^{-1}(g^{n-2}, g)
\cdots \sigma^{-1}(g, g) \, t_g^n \times \\
\times \alpha(g, g) \cdots \alpha(g^{n-2}, g) \, \alpha(g^{n-1},g) \, .
\end{multline}
Since the values of an invertible cocycle on grouplike elements are invertible
elements,  since $t_{g^n} = t_1$, and since $\sigma^{-1}(g, h) = 1/\sigma(g, h)$ 
for all grouplike elements $g,h$, Formula~\eqref{ggg} implies
\[
t_g^n = t_1 \, 
\frac{\sigma(g^{n-1}, g) \, \sigma(g^{n-2}, g)\cdots \sigma(g, g)}
{\alpha(g, g) \cdots \alpha(g^{n-2}, g) \, \alpha(g^{n-1},g)} \, .
\]
The right-hand side belongs to~$\BB_H^{\alpha}$.
It follows that $t_g$ is in~$A$ for each grouplike element of~$\Sigma$.

Since the grouplike elements of~$\Sigma$ are of finite order
and generate the group of grouplike elements of~$H$,
any grouplike element~$g$ of~$H$ is a product $g = g^{[1]} \cdots g^{[n]}$
of grouplike elements of~$\Sigma$
for which we have just established that $t_{g^{[1]}}, \ldots , t_{g^{[n]}}$
belong to~$A$. It then follows from Lemma~\ref{txtytz} and~\eqref{gitere}
that 
\begin{equation*}
t_{g^{[1]} \cdots g^{[n]}} 
= \kappa(g^{[1]}, \ldots, g^{[n]}) \,
t_{g^{[1]}} \, t_{g^{[2]}} \cdots  t_{g^{[n-1]}} \, t_{g^{[n]}} \, ,
\end{equation*}
where $\kappa(g^{[1]}, \ldots, g^{[n]})$ is the invertible element of~$\BB_H^{\alpha}$
given by
\begin{multline*}
\kappa(g^{[1]}, \ldots, g^{[n]})\\
= \frac{\alpha(g^{[1]}, g^{[2]}) \cdots\alpha(g^{[1]} \, g^{[2]} \cdots g^{[n-2]}, g^{[n-1]}) \, 
\alpha(g^{[1]} \, g^{[2]} \cdots g^{[n-1]}, g^{[n]})}{\sigma(g^{[1]} \, g^{[2]}\cdots g^{[n-1]}, g^{[n]}) \,
\sigma(g^{[1]} \, g^{[2]} \cdots g^{[n-2]}, g^{[n-1]}) \cdots \sigma(g^{[1]} , g^{[2]})} \, .
\end{multline*}
Therefore, $t_g \in A$ for every grouplike element of~$H$.

We next show that $t_x$ belongs to~$A$ for every skew-primitive element~$x$ of~$\Sigma$.
It is easy to check that if $x$ satisfies~\eqref{skew},
then for all $p\geq 2$, 
\begin{equation}\label{xitere}
\Delta^{(p)}(x) = g^{\otimes p} \otimes x +
\sum_{i=1}^{p-1} \, g^{\otimes (p-i)} \otimes x \otimes h^{\otimes i}  
+ x \otimes h^{\otimes p}  \, .
\end{equation}
Thus the iterated coproduct of any skew-primitive element~$x$
is a sum of tensor product of elements, all of which are grouplike, 
except for exactly one, which is~$x$.
It then follows from Lemma~\ref{txtytz} and~\eqref{xitere} that
for each~$n\geq 1$ the element~$t_{x^n}$ 
is a linear combination with coefficients in~$\BB_H^{\alpha}$
of monomials of the form $t_{g_1} t_{g_2} \cdots t_{g_{n-p}} t_x^p$,
where $0 \leq p \leq n$ and $g_1, \ldots, g_{n-p}$ are grouplike elements.
It is easily checked that in this linear combination
there is a unique monomial of the form $t_x^n$
whose coefficient is the invertible element of~$\BB_H^{\alpha}$
\[
\sigma^{-1}(g^{n-1}, g) \, \sigma^{-1}(g^{n-2}, g)\cdots \sigma^{-1}(g, g) \, 
\alpha(h, h) \cdots \alpha(h^{n-2}, h) \, \alpha(h^{n-1},h)\, .
\]
Since $t_g$ belongs to~$A$ for any grouplike element~$g\in H$,
it follows that, for all $n\geq 1$, the element~$t_{x^n}$
is a polynomial of degree~$n$ in~$t_x$ with coefficients in~$A$.
By hypothesis, there are scalars $\lambda_1, \ldots, \lambda_{n-1}, \lambda_n \in k$
for some positive integer~$n$ such that
\[
x^n + \lambda_1 x^{n-1} + \cdots + \lambda_{n-1} x + \lambda_n = 0 \, .
\]
Therefore, $t_x$ satisfies a degree~$n$ polynomial equation
with coefficients in the integral closure~$A$ and with highest-degree coefficient
equal to~$1$. 
This proves that $t_x \in A$.

To complete the proof, it suffices to check that
$t_z$ belongs to~$A$ for any product~$z$ of grouplike
or skew-primitive elements $x^{[1]}, \ldots , x^{[n]}$ 
such that $t_{x^{[1]}}, \ldots, t_{x^{[n]}}$ belong to~$A$.
It follows from Lemma~\ref{txtytz},~\eqref{gitere}, and~\eqref{xitere} that
$t_z$ is a linear combination with coefficients in~$\BB_H^{\alpha}$ 
of products of the variables $t_{x^{[1]}}, \ldots, t_{x^{[n]}}$ 
and of variables of the form~$t_g$, where $g$ is grouplike. 
Since these monomials belong to~$A$, so does~$t_z$.
\end{proof}

\section{How to construct elements of~$\BB_H^{\alpha}$}\label{identities}

In the example considered in Section~\ref{Sweedler} 
we reformulated the values of the generic
cocycle in terms of certain rational fractions $E$, $R$, $S$, $T$, $U$.
The aim of this last section is to explain how we found these fractions
by presenting a general systematic way 
of producing elements of~$\BB_H^{\alpha}$ for an arbitrary Hopf algebra.
To this end we introduce a new set of symbols.

\subsection{The symbols~$X_x$}\label{T(XH)}

Let $H$ be a  Hopf algebra and
$X_H$ a copy of the underlying vector space of~$H$;
we denote the identity map from $H$ to~$X_H$ by $x\mapsto X_x$ 
for all $x\in H$.

Consider the tensor algebra $T(X_H)$ of the vector space~$X_H$
over the ground field~$k$:
\[
T(X_H) = \bigoplus_{r\geq 0}\, X_H^{\otimes r} \, .
\]
If $\{x_i\}_{i\in I}$ is a basis of~$H$, 
then $T(X_H)$ is the free noncommutative algebra
over the set of indeterminates $\{X_{x_i}\}_{i\in I}$.

The algebra $T(X_H)$ is an $H$-comodule algebra 
equipped with the coaction $\delta : T(X_H) \to T(X_H) \otimes H$ given 
for all $x\in H$ by
\begin{equation}\label{TXcoaction2}
\delta(X_x) = \sum_{(x)}\, X_{x_{(1)}} \otimes x_{(2)} \, .
\end{equation}

\subsection{Co\"\i nvariant elements of~$T(X_H)$}

Let us now present a general method to construct co\"invariant elements of~$T(X_H)$.
We need the following terminology.

Given an integer $n\geq 1$, an \emph{ordered partition} of~$\{1, \ldots, n\}$
is a partition $\underline{I} = (I_1, \ldots, I_r)$ of~$\{1, \ldots, n\}$
into disjoint nonempty subsets $I_1, \ldots, I_r$ such that
$i < j$ for all $i \in I_k$ and $j\in I_{k+1}$ ($1 \leq k \leq r-1$).

If $x[1], \ldots, x[n]$ are $n$ elements of~$H$ and if
$I = \{i_1 < \cdots < i_p\}$ is a subset of~$\{1, \ldots, n\}$,
we set
$x[I] = x[i_1] \cdots x[i_p] \in H$.
If $\underline{I} = (I_1, \ldots, I_r)$ is an ordered partition of~$\{1, \ldots, n\}$,
then clearly
$x[I_1] \cdots x[I_r] = x[1] \cdots x[n]$.

Now let $x[1], \ldots, x[n]$ be $n$~elements of~$H$ and 
$\underline{I} = (I_1, \ldots, I_r)$, $\underline{J} = (J_1, \ldots, J_s)$
be ordered partitions of~$\{1, \ldots, n\}$.
We consider the following element of~$T(X_H)$:
\begin{equation}\label{P0}
P_{x[1], \ldots, x[n]; \underline{I}, \underline{J}} = 
\sum_{(x[1]), \ldots, (x[n])}\,
X_{x[I_1]_{(1)}} \cdots X_{x[I_r]_{(1)}} X_{S(x[J_s]_{(2)})} \cdots X_{S(x[J_1]_{(2)})}
\, .
\end{equation}
The element $P_{x[1], \ldots, x[n]; \underline{I}, \underline{J}}$ is 
an homogeneous element of~$T(X_H)$ of degree $r+s$.
Observe that $P_{x[1], \ldots, x[n]; \underline{I}, \underline{J}}$
is linear in each variable $x[1], \ldots, x[n]$.

We have the following generalization of~\cite[Lemma~2.1]{AK}.

\begin{Prop}\label{identity0}
Each element $P_{x[1], \ldots, x[n]; \underline{I}, \underline{J}}$
of~$T(X_H)$ is co\"invariant.
\end{Prop}

\pf
By \eqref{TXcoaction2}, 
$\delta(P_{x[1], \ldots, x[n]; \underline{I}, \underline{J}})$ is equal to
\begin{multline*}
\sum_{(x[1]), \ldots, (x[n])}\,
X_{x[I_1]_{(1)}} \cdots X_{x[I_r]_{(1)}} 
X_{S(x[J_s]_{(4)})} \cdots X_{S(x[J_1]_{(4)})} \\
\otimes  \; x[I_1]_{(2)} \cdots x[I_r]_{(2)} \, S(x[J_s]_{(3)}) \cdots S(x[J_1]_{(3)})
\end{multline*}
\begin{multline*}
=  \sum_{(x[1]), \ldots, (x[n])}\,
X_{x[I_1]_{(1)}} \cdots X_{x[I_r]_{(1)}}  X_{S(x[J_s]_{(4)})} \cdots X_{S(x[J_1]_{(4)})} \\
\otimes  x[1]_{(2)} \cdots x[n]_{(2)} \, S(x[n]_{(3)}) \cdots S(x[1]_{(3)}) 
\end{multline*}
\begin{multline*}
= \sum_{(x[1]), \ldots, (x[n])}\,
X_{x[I_1]_{(1)}} \cdots X_{x[I_r]_{(1)}}  X_{S(x[J_s]_{(3)})} \cdots X_{S(x[J_1]_{(3)})}  \\
\otimes  \eps(x[1]_{(2)}) \cdots \eps(x[n]_{(2)}) 
\end{multline*}
\[
=  \sum_{(x[1]), \ldots, (x[n])}\,
X_{x[I_1]_{(1)}} \cdots X_{x[I_r]_{(1)}}  X_{S(x[J_s]_{(2)})} \cdots X_{S(x[J_1]_{(2)})}
\otimes 1  \, .
\]
Therefore, $\delta(P_{x[1], \ldots, x[n]; \underline{I}, \underline{J}})
 = P_{x[1], \ldots, x[n]; \underline{I}, \underline{J}}  \otimes  1$
 and the conclusion follows.
\epf

As special cases of the previous proposition,
the following elements of $T(X_H)$ are co\"invariant
for all $x,y \in H$:
\begin{equation}\label{P1}
P_x = P_{x; (\{1\}),  (\{1\}) } = \sum_{(x)}\, X_{x_{(1)}} \, X_{S(x_{(2)})} 
\end{equation}
and
\begin{equation}\label{P2}
P_{x,y} = P_{x,y; (\{1\}, \{2\}),  (\{1,2\}) } = \sum_{(x), (y)}\,
X_{x_{(1)}} \, X_{y_{(1)}} \, X_{S(x_{(2)} y_{(2)})} \, . 
\end{equation}

\subsection{The generic evaluation map}\label{univ-eval-map}

As in Section~\ref{generic cocycle}, let $H$ be a Hopf algebra,
$\alpha : H \times H \to k$ a normalized invertible cocycle,
and ${}^{\alpha} H$ the corresponding twisted algebra.

Consider the algebra morphism
$\mu_{\alpha} : T(X_H) \to  S(t_H) \otimes {}^{\alpha} H$
defined for all $x\in H$ by
\begin{equation*}
\mu_{\alpha}(X_x) =  \sum_{(x)}\, t_{x_{(1)}} \otimes u_{x_{(2)}} \, .
\end{equation*}

The morphism $\mu_{\alpha}$ possesses the following properties
(see \cite[Sect.~4]{AK}).

\begin{Prop}\label{lem-mu-univ}
(a) The morphism $\mu_{\alpha} : T(X_H) \to  S(t_H) \otimes {}^{\alpha} H$ is an
$H$-comodule algebra morphism.

(b) If the ground field $k$ is infinite, then for every $H$-comodule algebra morphism
$\mu : T(X_H) \to {}^{\alpha} H$, 
there is a unique algebra morphism 
$\chi : S(t_H) \to k$
such that 
\[
\mu = (\chi \otimes \id) \circ \mu_{\alpha} \, .
\]
\end{Prop}

In other words, any $H$-comodule algebra morphism
$\mu : T(X_H) \to {}^{\alpha} H$ is obtained by specialization
from~$\mu_{\alpha}$.
For this reason we call $\mu_{\alpha}$ the 
\emph{generic evaluation map} for~${}^{\alpha} H$.

Now we have the following result (see~\cite[Sect.~8]{AK}).

\begin{Prop}
If $P \in T(X_H)$ is co\"invariant,
then $\mu_{\alpha}(P)$ belongs to~$\BB_H^{\alpha}$.
\end{Prop}

It follows that the image 
$\mu_{\alpha}(P_{x[1], \ldots, x[n]; \underline{I}, \underline{J}})$
of all co\"invariant elements defined by~\eqref{P0} belong to~$\BB_H^{\alpha}$.
This provides a systematic way to produce elements of~$\BB_H^{\alpha}$.

\begin{example}
When $H= H_4$ is the Sweedler algebra, 
it is easy to check that the elements 
$R$, $S$, $T$, $U$ of~\eqref{ERSTU} are 
obtained in this way: we have
\[
R = \mu_{\alpha}(P_x) \, , \quad
T = \mu_{\alpha}(P_{y - z}) \, , \quad
U = \mu_{\alpha}(P_{x,z}) \, , \quad 
ES = \mu_{\alpha}(P_{y,y}) \, ,
\]
where $\{1,x,y,z\}$ is the basis of~$H_4$ defined in Section~\ref{Sweedler}
and $P_x$, $P_{y - z}$, $P_{x,z}$, and~$P_{y,y}$
are special cases of the noncommutative polynomials 
defined by~\eqref{P1} and~\eqref{P2}.
\end{example}

\begin{remark}
In~\cite{AK} we developped a theory of \emph{polynomial identities}
for $H$-comodule algebras. This theory applies in particular
to the twisted algebras~${}^{\alpha} H$. 
We established that the
$H$-identities of~${}^{\alpha} H$, as defined in \emph{loc.\ cit}., 
are exactly the elements of~$T(X_H)$ that lie in the kernel of
the generic evaluation map~$\mu_{\alpha}$.
Thus the $H$-comodule algebra 
\[
\UU_H^{\alpha} = T(X_H)/\Ker \mu_{\alpha}
\]
plays the r\^ole of a \emph{universal comodule algebra}. 
We also constructed an $H$-comodule algebra morphism 
$\UU_H^{\alpha} \to \AA_H^{\alpha}$;
under certain conditions
this map turns the generic Galois extension~$\AA_H^{\alpha}$
into a \emph{central localization} of the universal comodule algebra~$\UU_H^{\alpha}$
(see~\cite[Sect.~9]{AK} for details).
\end{remark}

\bibliographystyle{amsplain}

\begin{thebibliography}{99}



\bibitem{AHN} E.~Aljadeff, D.~Haile, M.~Natapov,
\textit{Graded identities of matrix algebras and the universal graded algebra},
arXiv:0710.5568, Trans.\ Amer.\ Math.\ Soc.\ (2009), in press.

\bibitem{AK} E.~Aljadeff, C.~Kassel,
\textit{Polynomial identities and noncommutative versal torsors},
Adv.\ Math.\ 218 (2008), 1453--1495, 
doi:10.1016/j.aim.2008.03.014.

\bibitem{AN} E.~Aljadeff, M.~Natapov,
\textit{On the universal $G$-graded central simple algebra},
Actas del XVI Coloquio Latinoamericano de \'Algebra
(eds.\ W.~Ferrer Santos, G.~Gonzalez-Sprinberg, A.~Rittatore, A.~Solotar),
Biblioteca de la Revista Matem\'atica Iberoamericana, Madrid~2007. 

\bibitem{Au} T. Aubriot,
\textit{On the classification of Galois objects over the 
quantum group of a nondegenerate bilinear form},
Manuscripta Math.\ 122 (2007), 119--135. 

\bibitem{Bi} J. Bichon, 
\textit{Galois and bigalois objects over monomial non-semisimple Hopf algebras},
J.~Algebra Appl.\ 5 (2006), 653--680. 

\bibitem{BCM} R. J. Blattner, M. Cohen, S. Montgomery, 
\textit{Crossed products and inner actions of Hopf algebras},
Trans.\ Amer.\ Math.\ Soc.\ 298 (1986), 671--711. 

\bibitem{BM} R. J. Blattner, S. Montgomery, 
\textit{A duality theorem for Hopf module algebras},
J.~Algebra 95 (1985), 153--172.

\bibitem{Doi} Y. Doi, 
\textit{Equivalent crossed products for a Hopf algebra},
Comm.\ Algebra 17 (1989), 3053--3085.

\bibitem{DT} Y.~Doi, M.~Takeuchi, 
\textit{Cleft comodule algebras for a bialgebra},
Comm.\ Algebra 14 (1986), 801--817.

\bibitem{DT2} Y.~Doi, M.~Takeuchi, 
\textit{Quaternion algebras and Hopf crossed products}, 
Comm.\ Algebra 23 (1995), 3291--3325.

\bibitem{GMS} S.~Garibaldi, A.~Merkurjev, J.-P.\ Serre, 
\textit{Cohomological invariants in Galois coho\-mol\-ogy}, 
Univ.\ Lecture Ser.~28, Amer.\ Math.\ Soc., Providence, RI, 2003.

\bibitem{Gu} R. G\"unther, 
\textit{Crossed products for pointed Hopf algebras},
Comm.\ Algebra 27 (1999), 4389--4410. 

\bibitem{K} C.~Kassel, 
\textit{Quantum principal bundles up to homotopy equivalence},
The Legacy of Niels Henrik Abel, The Abel Bicentennial, Oslo, 2002,
O.~A.~Laudal, R.~Piene (eds.), Springer-Verlag 2004, 737--748
(see also arXiv:math.QA/0507221).

\bibitem{KS} C.~Kassel, H.-J.\ Schneider,
\textit{Homotopy theory of Hopf Galois extensions},
Ann.\ Inst.\ Fourier (Grenoble) 55 (2005), 2521--2550.

\bibitem{Ma1} A. Masuoka,
\textit{Cleft extensions for a Hopf algebra generated by a nearly primitive element},
Comm.\ Algebra 22 (1994), 4537--4559.

\bibitem{Ma11} A. Masuoka, 
\textit{Cocycle deformations and Galois objects for some cosemisimple Hopf algebras of finite dimension},
New trends in Hopf algebra theory (La Falda, 1999), 195--214,
Contemp.\ Math., 267, Amer.\ Math.\ Soc., Providence, RI,~2000. 

\bibitem{Ma2} A. Masuoka,
\textit{Abelian and non-abelian second cohomologies of quantized enveloping algebras}, 
J.~Algebra 320 (2008), 1--47.

\bibitem{M2} S. Montgomery, 
\textit{Hopf algebras and their actions on rings},
CBMS Conf.\ Series in Math., vol.~82, Amer.\ Math.\ Soc., Providence, RI, 1993.

\bibitem{PvO} F. Panaite, F.~Van Oystaeyen,
\textit{Clifford-type algebras as cleft extensions for some pointed Hopf algebras},
Comm.\ Algebra 28 (2000), 585--600.

\bibitem{Sb1} P. Schauenburg, 
\textit{Hopf bi-Galois extensions},
Comm.\ Algebra  24 (1996), 3797--3825.

\bibitem{Sb2} P. Schauenburg, 
\textit{Galois objects over generalized Drinfeld doubles, 
with an application to $u_q(\gs\gl_2)$},
J.~Algebra  217 (1999), 584--598.

\bibitem{S} H.-J.\ Schneider, 
\textit{Principal homogeneous spaces for arbitrary Hopf algebras}, 
Israel J.\ Math.\ 72 (1990), 167--195.

\bibitem{Sw} M.~Sweedler, 
\textit{Hopf algebras}, 
W. A. Benjamin, Inc., New York,~1969.




\end{thebibliography}

\end{document}